\documentclass[12pt]{amsart}

\usepackage{amssymb,epsfig,amsfonts,setspace,dsfont}
\usepackage{amsmath,amscd,amsthm,array,geometry}
\usepackage{xcolor}

\newtheorem{thm}{Theorem}[section]

\newtheorem{prop}[thm]{Proposition}
\newtheorem{conj}[thm]{Conjecture}
\newtheorem{cor}[thm]{Corollary}

\theoremstyle{definition}
\newtheorem{defn}[thm]{Definition}
\newtheorem{rem}[thm]{Remark}
\newtheorem{exam}[thm]{Example}

\newcommand{\GL}{\mathrm{GL}}
\newcommand{\SL}{\mathrm{SL}}

\newcommand{\FF}{\mathbb{F}}
\newcommand{\ZZ}{\mathbb{Z}}
\newcommand{\CC}{\mathbb{C}}
\newcommand{\QQ}{\mathbb{Q}}
\newcommand{\TT}{\mathbb{T}}
\newcommand{\frob}{\mathrm{Frob}}
\newcommand{\End}{\mathrm{End}}

\newcommand{\rar}{\rightarrow}
\newcommand{\pd}{\text{.}}
\newcommand{\cm}{\text{,}}

\begin{document}

\title[Detecting simple Hecke modules via congruences]{Detecting large simple rational Hecke modules for $\Gamma_0(N)$ via congruences}

\author{Michael Lipnowski}\address{Michael Lipnowski, Department of Mathematics, University of Toronto.}\email{malipnow@math.utoronto.edu}
\author{George J. Schaeffer}\address{George J. Schaeffer, Department of Mathematics, Stanford University.}\email{gschaeff.research@gmail.com}

\begin{abstract}We describe a novel method for bounding the dimension $d$ of the largest simple Hecke submodule of $S_2(\Gamma_0(N);\QQ)$ from below. Such bounds are of interest because of their relevance to the structure of $J_0(N)$, for instance.

In contrast with previous results of this kind, our bound does not rely on the equidistribution of Hecke eigenvalues. Instead, it is obtained via a Hecke-compatible congruence between the target space and a space of modular forms whose Hecke eigenvalues are easily controlled. For prime levels $N\equiv 7\mod 8$ our method yields an unconditional bound of $d\ge\log_2\log_2(\frac{N}{8})$, improving the known bound of $d\gg\sqrt{\log\log N}$ due to Murty--Sinha and Royer. We also discuss conditional bounds, the strongest of which is $d\gg_\epsilon N^{1/2-\epsilon}$ over a large set of primes $N$, contingent on Soundararajan's heuristics for the class number problem and Artin's conjecture on primitive roots.

We also propose a number of Maeda-style conjectures based on our data, and we outline a possible congruence-based approach toward the conjectural Hecke simplicity of $S_k(\SL_2(\ZZ);\QQ)$.\end{abstract}

\maketitle

\section{Introduction}

\subsection{Maeda's conjecture and asymptotic simplicity}\label{Intro}

It is natural to ask how spaces of modular forms decompose under the actions of their Hecke algebras.\footnote{Throughout this article, when $V$ is a space of modular forms of level $N$, then by ``the Hecke algebra of $V$'' we always mean the (commutative) subalgebra $\TT(V)$ of $\End(V)$ generated by the Hecke operators $T_n$ for $n$ {\it relatively prime} to $N$. This is sometimes referred to as the {\it anemic} Hecke algebra. If $V$ is clear from context we may denote its Hecke algebra simply by $\TT$.

} Maeda's conjecture attempts to answer this question in the level $1$ case: It posits that for all even integers $k\ge 4$ the $\QQ$-vector space of weight $k$ cusp forms for $\Gamma(1)=\SL_2(\ZZ)$, here denoted $S_k(\Gamma(1);\QQ)$, is a {\it simple} Hecke module. In fact, Maeda conjectured that this space is simple under the action of $T_2$.

Interesting arithmetic applications of Maeda's conjecture have been noted by many authors: See Hida--Maeda \cite{HM}, Conrey--Farmer \cite{CF}, and Kohnen--Zagier \cite{KZ} for a small sample. Computational evidence for Maeda's conjecture has been given by Buzzard \cite{Buzzard}, Farmer--James \cite{FJ}, and Ghitza--McAndrews \cite{Ghitza}, among others. At the time of writing, Maeda's conjecture has been verified for weights $k\le 14000$.


Another well-studied sequence of $\QQ$-vector spaces of modular forms is
\begin{equation}\tag{$*$}\label{interestingsequence}\{\,S_2(\Gamma_0(N);\QQ):\text{$N$ prime}\,\}\text{.}\end{equation}
The Hecke module decompositions of the terms in this sequence have attracted interest because of established bijective correspondences between

\begin{itemize}\item[i.]Simple Hecke submodules $V'\subseteq S_2(\Gamma_0(N);\QQ)$ such that $\dim_{\QQ}(V')=d$;
\item[ii.]$\mathrm{Gal}(\overline{\QQ}/\QQ)$-orbits $\mathcal{O}$ of normalized cuspidal eigenforms in $S_2(\Gamma_0(N);\overline{\QQ})$ such that $|\mathcal{O}|=d$; and
\item[iii.]Simple isogeny factors $A$ of $J_0(N)$, the Jacobian of the modular curve $X_0(N)$, such that $\dim_{\QQ}(A)=d$.\end{itemize}


In contrast with what Maeda predicts in the level $1$ setting, the terms in the sequence \eqref{interestingsequence} are {\it almost never} simple Hecke modules. The principal obstruction to simplicity in this case is the existence of the level $N$ {\it Atkin--Lehner involution} $w_N:f(\tau)\mapsto f(-\frac{1}{N\tau})$ on $S_2(\Gamma_0(N);\QQ)$  \cite{AL}. Since the nebentypus is trivial, $w_N$ commutes with the Hecke operators, so the $\pm1$-eigenspaces of $w_N$ are themselves Hecke submodules of $S_2(\Gamma_0(N);\QQ)$ and we obtain a decomposition
\[S_2(\Gamma_0(N);\QQ)=S_2(\Gamma_0(N);\QQ)^+\oplus S_2(\Gamma_0(N);\QQ)^-\]
of {\it Hecke modules}. Both factors in this decomposition are nontrivial for $N$ sufficiently large.\footnote{In fact, both factors in this decomposition are nontrivial as long as $$N \neq 2,3,5,7,11,13,17,19,29,31,41,47,59,71\pd$$} The difference $\dim_{\QQ}S_2(\Gamma_0(N);\QQ)^+ -\dim_{\QQ}S_2(\Gamma_0(N);\QQ)^-$ can be expressed explicitly in terms of class numbers of imaginary quadratic fields, and well-known class number bounds then imply that this difference is $\ll_\epsilon N^{1/2+\epsilon}$. It follows that
\[\lim_{N\rightarrow\infty}\left(\frac{\dim_\QQ S_2^+(\Gamma_0(N);\QQ)}{\dim_\QQ S_2(\Gamma_0(N);\QQ)}\right)=\lim_{N\rightarrow\infty}\left(\frac{\dim_\QQ S_2^-(\Gamma_0(N);\QQ)}{\dim_\QQ S_2(\Gamma_0(N);\QQ)}\right)=\frac{1}{2}\text{.}\]


By way of illustration, we tabulate the dimensions of the Hecke-simple factors of $S_2(\Gamma_0(N);\QQ)^{\pm}$ for prime levels $N \in [10007,10103]$ below (data over a broader range are given in Table 3 of \S \ref{almostsimplicity}): 

\begin{center}\begin{tabular}{|l|c|c||l|c|c|}\hline $N$ & $S_2(\Gamma_0(N);\QQ)^+$ & $S_2(\Gamma_0(N);\QQ)^-$ & $N$ & $S_2(\Gamma_0(N);\QQ)^+$ & $S_2(\Gamma_0(N);\QQ)^-$\\\hline\hline
$10007$ & $[455]$ & $[379]$ & $10069$ & $[432]$ & $[406]$\\\hline
$10009$ & $[440]$ & $[393]$ & $10079$ & $[486][1]$ & $[353]$\\\hline
$10037$ & $[438]$ & $[398]$ & $10091$ & $[458][1]$ & $[382]$\\\hline
$10039$ & $[452]$ & $[384]$ & $10093$ & $[430]$ & $[410]$\\\hline
$10061$ & $[454]$ & $[383][1]$ & $10099$ & $[443][1][1]$ & $[396]$\\\hline
$10067$ & $[440]$ & $[399]$ & $10103$ & $[446]$ & $[394][2]$\\\hline
\end{tabular}\smallskip

{\sc Table 1.} Dimensions of simple Hecke submodules of $S_2(\Gamma_0(N);\QQ)$.

\end{center}

While $S_2(\Gamma_0(N);\QQ)$ itself is almost never a simple Hecke module, data such as those in the table above suggest that the Atkin--Lehner eigenspaces $S_2(\Gamma_0(N);\QQ)^\pm$ are {\it often (but not always)} simple. When they are {\it not} simple, it is due to the existence of small sporadic simple factors corresponding via modularity to elliptic curves and other low-dimensional abelian varieties with good reduction away from $N$.  In Table 1, the $1$-dimensional simple Hecke submodules correspond to elliptic curves of conductor equal to the level; the $2$-dimensional Hecke submodule at level 10103 corresponds to an abelian surface of conductor $10103^2$ having multiplication by an order in $\mathrm{Mat}_2\,\QQ(\sqrt{5})$. In any case, one observes in the table that both $S_2(\Gamma_0(N);\QQ)^+$ and $S_2(\Gamma_0(N);\QQ)^-$ always contain a {\it large} simple Hecke submodule. Further computation seems to indicate that this ``almost-simplicity'' phenomenon continues as $N$ grows.


In addition to the two sequences discussed so far, there are other sequences of $\QQ$-vector spaces of modular forms whose terms seem to consistently admit large simple Hecke submodules. To ease notation, if $V$ is a vector space over a field $F$ and $A$ is an $F$-subalgebra of $\End(V)$, we set
\[\dim_{F,A}^*V=\max\{\,\dim_ F V'\,:\,\text{$V'$ is an $ F$-subspace of $V$ and a simple $A$-module}\,\}\]
For example, according to Table 1, $\dim_{\QQ,\TT}^*S_2(\Gamma_0(10061);\QQ)^-=383$.

\begin{defn}\label{asymptoticsimplicity}Let $\{V_i\}_i$ be a sequence of vector spaces over a fixed field $ F$ and for each $i$, let $A_i$ be an $ F$-subalgebra of $\End(V_i)$. We say that $\{V_i\}_i$ is {\it asymptotically simple} (with respect to $\{A_i\}_i$) if
\[\lim_{i\rightarrow\infty}\left(\frac{\dim_{ F,A_i}^*V_i}{\dim_{ F}V_i}\right)=1\text{.}\]


\end{defn}

Maeda's conjecture implies that the sequence $\{S_k(\Gamma(1);\QQ)\}_k$ is asymptotically simple (with respect to the terms' Hecke algebras). We conjecture further that any sequence of spaces of modular forms with primitive nebentypen and with either the weight or the level tending to infinity is asymptotically Hecke-simple.

\begin{conj}\label{primitivecharacterconj}Let $\{(k_i,N_i,\chi_i)\}_i$ be a sequence of triples such that all $k_i$ and $N_i$ are positive integers and $\chi_i:(\ZZ/N_i\ZZ)^\times\rightarrow\CC^\times$ is primitive character satisfying $\chi_i(-1)=(-1)^{k_i}$. Let $L=\QQ(\{\chi_i\}_i)$.

\begin{itemize}\item[a.] If there is a fixed positive integer $N$ and a fixed primitive $\chi:(\ZZ/N\ZZ)^\times\rightarrow\CC^\times$ such that $(N_i,\chi_i)=(N,\chi)$ for all $i$, and $\{k_i\}_i$ is an increasing sequence; or
\item[b.] If there is a fixed integer $k\ge 2$ such that $k_i=k$ for all i, $\{N_i\}_i$ is an increasing sequence, and $L/\QQ$ is a finite extension;\end{itemize}
then the sequence $\{S_{k_i}(N_i,\chi_i;L)\}_i$ of Hecke modules over $L$ is asymptotically simple.\end{conj}

We also conjecture that the Atkin--Lehner eigenspaces of $S_k(\Gamma_0(N);\QQ)$ for fixed even weights $k\ge 2$ and varying prime levels $N$ are asymptotically simple.

\begin{conj} \label{fixedweighttrivialnebentypusincreasinglevel}
Let $k$ be even and $k\ge 2$. The sequences\[\{\,S_k(\Gamma_0(N);\QQ)^+:\text{$N$ prime}\,\}\quad\text{and}\quad\{\,S_k(\Gamma_0(N);\QQ)^-:\text{$N$ prime}\,\}\]of Hecke modules over $\QQ$ are asymptotically simple.
\end{conj}

While we avoid formal generalizations for brevity, the above conjectures can be adapted appropriately to cases where $N$ is composite squarefree and/or the nebentypen are imprimitive, if one accounts for Hecke module decompositions induced by the Atkin--Lehner involutions.

We should add that Tsaknias \cite{T} and Dieulefait--Tsaknias \cite{DT} have recently proposed conjectures that predict an exact count of the Galois orbits of normalized eigenforms in $S_k(\Gamma_0(N);\overline{\QQ})^{\mathrm{new}}$ for {\it arbitrary} fixed $N$ and $k$ sufficiently large. Their prediction in the case of prime $N$ is that there are {\it two} Galois orbits as $k\rightarrow\infty$ (corresponding to the two Atkin--Lehner eigenspaces), which is certainly compatible with our Conjecture \ref{fixedweighttrivialnebentypusincreasinglevel}.

\subsection{Main results and methods}Though we know no general method for constructing simple Hecke modules of the magnitude required to approach Conjectures \ref{primitivecharacterconj} and \ref{fixedweighttrivialnebentypusincreasinglevel}, our main result, Theorem \ref{largesimplesubmoduleviaweight1proof}, certainly represents progress. It is perhaps especially encouraging in light of the conditional bounds given in Corollary \ref{smorgasbordofconjectures}.

Some notation: If $k\ge 0$ and $N\ge 1$ are integers, $F$ is a field, and $\chi:(\ZZ/N\ZZ)^\times\rightarrow F^\times$ is a character, then by $M_k(N,\chi; F)$ we mean the $ F$-vector space of weight $k$ {\it Katz modular forms} of level $N$ and character $\chi$ \cite{Katz}. Its subspace of cusp forms is denoted $S_k(N,\chi; F)$.\footnote{When $\mathrm{char}\,F=0$, the space of Katz modular forms over $F$ can be obtained from the corresponding space over $\QQ$ via extension of scalars.  However, not all weight $1$ Katz modular forms over $\FF_p$ can be obtained by mod-$p$ reduction of classical modular forms with coefficients in $\ZZ$ \cite[\S4]{Khare} \cite[\S5.1]{Schaeffer}.}


When $R$ is a subring of $F$ and $N$ is invertible in $R$, then $S_k(N,\chi;R)$ as defined in \cite{Katz} coincides with the $R$-submodule of $S_k(N,\chi;F)$ consisting of those forms whose $q$-expansions land in $R[[q]]$.

When $\chi$ is trivial, we will usually replace the $(N,\chi)$ in our notation for spaces of modular forms with the congruence subgroup $\Gamma_0(N)$. Finally, in what follows, $\varepsilon_N$ will denote the quadratic character of conductor $N$.

\begin{thm} \label{largesimplesubmoduleviaweight1proof}
Let $N$ be squarefree with $N\equiv 3\mod 4$ so that $\varepsilon_N(-1)=-1$, and let $p$ be a prime not dividing $N$. Set
\[d_p(N)=\left\{\begin{array}{ll}\dim_{\QQ,\TT}^*S_2(\Gamma_0(N);\QQ)&\text{when $p=2$, and}\\\dim_{\QQ,\TT}^*S_p(N,\varepsilon_N;\QQ)&\text{when $p$ is odd.}\end{array}\right.\]

Suppose that the class group of $\QQ(\sqrt{-N})$ has a cyclic quotient of order $r$, where $r$ is not divisible by $p$, and let $m$ be the order of $p$ in $(\ZZ/r\ZZ)^\times/\{\pm 1\}$. If $m>2$, then $d_p(N)\ge m$.

\end{thm}
%
%
%
%



We prove Theorem \ref{largesimplesubmoduleviaweight1proof} in \S \ref{largegaloisorbits}. Roughly, the argument proceeds as follows: Let $H_p$ denote the {\it mod $p$ Hasse invariant}, which is an element of $M_{p-1}(\Gamma(1);\FF_p)$. Multiplication by $H_p$ yields an injective map
\[[H_p]:S_1(N,\tilde{\varepsilon}_N;\FF_p)\rightarrow S_p(N,\tilde{\varepsilon}_N;\FF_p):f\mapsto H_pf\]
that is {\it Hecke-compatible} in the sense that $[H_p]\circ T=T\circ [H_p]$ for all $T\in\TT$. Note that when $p=2$, $\tilde{\varepsilon}_N$ {\it is} the trivial character of level $N$, so $[H_2]:S_1(\Gamma_0(N);\FF_2)\rightarrow S_2(\Gamma_0(N);\FF_2)$.

Next, we consider the dihedral newforms in $S_1(N,\varepsilon_N;\overline{\QQ})$. With notation as in the theorem, there is a dihedral newform with $\zeta_r+\zeta_r^{-1}$ as a Hecke eigenvalue, and an explicit description of how the rational prime $p$ factors in the extension $\QQ(\zeta_r+\zeta_r^{-1})/\QQ$ allows us to produce a simple $\TT$-submodule of $S_1(N,\tilde{\varepsilon}_N;\FF_p)$ having dimension at least $m$. By applying the map $[H_p]$ we see that $S_p(N,\tilde{\varepsilon}_N;\FF_p)$ must also have a simple $\TT$-submodule of dimension at least $m$ and lifting back to characteristic zero yields the theorem.

In summary, we obtain the bound of Theorem \ref{largesimplesubmoduleviaweight1proof} by combining the inequality
\[\dim_{\FF_p,\TT}^*S_1(N,\tilde{\varepsilon}_N;\FF_p)\le\dim_{\FF_p,\TT}^*S_p(N,\tilde{\varepsilon}_N;\FF_p)\le d_p(N)\]
with the control over $\dim_{\FF_p,\TT}^*S_1(N,\tilde{\varepsilon}_N;\FF_p)$ furnished by dihedral forms.\footnote{The notation $\dim^*$ is given before Definition \ref{asymptoticsimplicity}.}




\begin{rem}As mentioned in the discussion of \S\ref{Intro}, simple Hecke-submodules of $S_2(N;\QQ)$ of dimension $d$ correspond bijectively with $d$-dimensional simple isogeny factors of the Jacobian $J_0(N)$ of the modular curve $X_0(N)$. So as to avoid repetition later on, we emphasize here:
\begin{quote}
\emph{Any lower bound on $d_2(N)$ is also a lower bound on the dimension of the largest simple isogeny factor of $J_0(N)$.}
\end{quote}
That is, though Theorem \ref{largesimplesubmoduleviaweight1proof} and Corollaries \ref{classnumberbound}--\ref{smorgasbordofconjectures} are stated in the language of Hecke modules, they are also provide information about the structure of $J_0(N)$.

\end{rem}

\subsection{Unconditional bounds on $d_2(N)$.}Next we discuss how one can use Theorem \ref{largesimplesubmoduleviaweight1proof} to obtain unconditional lower bounds on $d_p(N)$. Since these bounds are arguably most interesting when $p=2$, we will concentrate on this case.

Serre showed in \cite{Serre} that $d_2(N)$ is unbounded. His argument, based on the equidistribution of the eigenvalues of $T_\ell$ acting on $S_2(N,\CC)$ for fixed prime $\ell$ and increasing $N$, was made effective by Murty--Sinha \cite{MS} and Royer \cite{Royer}, establishing an asymptotic bound of $d_2(N)\gg\sqrt{\log\log N}$ as $N\rightarrow\infty$.

Theorem \ref{largesimplesubmoduleviaweight1proof} allows us to improve this bound for certain values of $N$:

%

\begin{cor}\label{classnumberbound}
With the notation and assumptions of Theorem \ref{largesimplesubmoduleviaweight1proof}, we have the lower bound $d_2(N)\ge\log_2(r-1)$.
\end{cor}

\begin{proof}The order of $2$ in $(\ZZ/r\ZZ)^\times/\{\pm1\}$ is easily seen to be at least $\log_2(r-1)$. The result follows from Theorem \ref{largesimplesubmoduleviaweight1proof}.\end{proof}


\begin{cor} \label{logloglowerboundsmallsplitprime}
If, in addition to the assumptions of Theorem \ref{largesimplesubmoduleviaweight1proof}, we suppose that $N$ is prime and $N\equiv 7\mod 8$, then $d_2(N)\ge\log_2\log_2 \left( \frac{N}{8} \right)$.
\end{cor}

\begin{proof}
Because $N \equiv 7$ mod 8:
\begin{itemize}
\item[a.]The discriminant of $K=\QQ(\sqrt{-N})$ is $-N$ and the ring of integers $\ZZ_K$ is equal to $\ZZ\left[\tfrac{1+\sqrt{-N}}{2}\right]$.
\item[b.]In light of (a.), $N$ is the only prime ramified in $K/\QQ$, so by genus theory, the class group of $K$ has {\it odd order}.
\item[c.]Since $2$ is a quadratic residue modulo $N$, the (unramified) prime $2$ splits in $K/\QQ$, say $2\ZZ_K=\mathfrak{p}\mathfrak{p}'$.
\end{itemize}

Following an argument of Ellenberg--Venkatesh \cite{EV}, $\mathfrak{p}$ generates a large subgroup of the class group: Let $r$ be the order of $\mathfrak{p}$ in $\mathrm{Cl}\,K$, so
\[\mathfrak{p}^r=\bigl(x+y\sqrt{-N}\bigr)\cdot\ZZ_K\]
for some $x,y\in\frac{1}{2}\ZZ$. Taking norms yields $2^r=x^2+Ny^2$. Since $r$ is odd by (b.)\ above, $2^r\ne x^2$, so we must have $y\ne 0$. Finally, because $y^2\ge\frac{1}{4}$ we have $2^r\ge\tfrac{N}{4}$ and therefore  $r\ge\log_2\bigl(\frac{N}{4}\bigr)$. The result now follows from the preceding corollary.\end{proof}

%


\subsection{Conditional bounds on $d_2(N)$.}

Recall the following well-known heuristics and conjectures due to Cohen-Lenstra (CL) \cite{CL}, Soundararajan (S) \cite{Sound}, and Artin (A) \cite{Artin} respectively:

\begin{itemize}\item (CL) For an effective positive proportion of $N \in [X,2X]$, the imaginary quadratic field $\QQ(\sqrt{-N})$ admits a cyclic quotient of odd	 order $r \gg_\epsilon N^{1/2-\epsilon}$.


\item (S) There are positive constants $a$ and $b$, not depending on $r$, such that the number $n_r$ of imaginary quadratic fields with {\it prime} class number $r$ satisfies $ar<n_r\log r<br$
\item (A) The element $2$ generates $(\ZZ/r\ZZ)^\times$ for an effective positive density of primes $r$.\footnote{Artin's primitive root conjecture was proven by Hooley conditional on the Generalized Riemann Hypothesis \cite{Hooley}.}
\end{itemize}

Under various combinations of (CL), (S), and (A), Theorem \ref{largesimplesubmoduleviaweight1proof} implies much stronger lower bounds for $d_2(N)$ than the unconditional lower bound given in Corollary \ref{logloglowerboundsmallsplitprime}.
\begin{cor} \label{smorgasbordofconjectures}
\begin{itemize}\item[a.] Under {\rm (CL)}, 
\[d_2(N) \gg \log N \]\
for a positive proportion of squarefree integers $N \in [X,2X]$.\footnote{Assuming (CL), our main Theorem implies that $d_2(N)$ is at least half the order of 2 in $(\mathbb{Z} / r \mathbb{Z})^\times$ for some integer $r \gg_\epsilon N^{1/2 - \epsilon}$.  This will typically have order of magnitude \emph{much} larger than $\log N$.  The lower bound for $d_2(N)$ from the above item is extremely wasteful but easily stated.}
\item[b.] Under {\rm(S)}, 
\[d_2(N) \gg_\epsilon N^{1/4-\epsilon}\]
for $N$ lying in a subset of the squarefree integers in $[X,2X]$ of size $\gg_\epsilon X^{1-\epsilon}$.
\item[c.] Under {\rm (S)$\,+\,$(A)},
\[d_2(N) \gg_\epsilon N^{1/2-\epsilon}\]
for $N$ lying in a subset of the squarefree integers in $[X,2X]$ of size $\gg_\epsilon X^{1-\epsilon}$.
\end{itemize}
\end{cor}

To put Corollary \ref{smorgasbordofconjectures} into perspective, note that any lower bound on $d_2(N)$ obtained via Theorem \ref{largesimplesubmoduleviaweight1proof} can never exceed the class number of $\QQ(\sqrt{-N})$.  Since this class number is at least $ \gg_\epsilon N^{1/2 - \epsilon}$ and at most $\ll_\epsilon N^{1/2 + \epsilon}$, the conditional asymptotic bound in (c.)\ above would essentially be the optimal bound accessible to our congruence method.

\section{Consequences of asymptotic simplicity for the structure of $J_0(N)$}

We next describe an application of Conjecture \ref{fixedweighttrivialnebentypusincreasinglevel} to the arithmetic of $J_0(N)$.

Brumer conjectured in \cite{Brumer} that the rank of $J_0(N)(\QQ)$ should be asymptotic to
\[\dim_{\QQ} S_2(\Gamma_0(N);\QQ)^+ \sim \tfrac{1}{2} \dim_{\QQ} S_2(\Gamma_0(N);\QQ)\]as $N\rar\infty$. The $L$-function of $J_0(N)$ factors as a product of $L(f,s)$ as $f$ varies over the weight 2 level $N$ cuspidal eigenforms for $\Gamma_0(N)$.  Brumer's conjecture is ``minimalist" in the sense that if $L(f,s)$ vanished at the central point to the smallest order allowed by the sign of its functional equation and we assume the Birch--Swinnerton-Dyer (BSD) Conjecture, then there would be an \emph{exact formula}\footnote{The functional equation for $L(f,s)$ has sign $-\epsilon$ where $\epsilon$ is the Atkin-Lehner eigenvalue of $f$.  This is the reason $S_2(\Gamma_0(N);\QQ)^+$ appears in formula \eqref{minimalistrank} instead of $S_2(\Gamma_0(N);\QQ)^-$.}
\begin{equation} \label{minimalistrank}
\mathrm{rank}\, J_0(N)(\QQ) = \dim_{\QQ} S_2(\Gamma_0(N);\QQ)^+\pd
\end{equation}

Assuming the data in Table 1 is representative, Conjecture \ref{fixedweighttrivialnebentypusincreasinglevel} should imply \eqref{minimalistrank} for many prime $N$.  Significantly more data concerning the small sporadic simple Hecke-submodules of $S_2(N;\QQ)$ appears in \S \ref{almostsimplicity}.

\begin{prop}[Exact formula for the rank of $J_0(N)$ for some $N$] \label{heckeirreducibilityminimalistconjecture}
Assume BSD for $J_0(N)$ and the Riemann Hypotheses for modular forms of weight 2 and for Dirichlet characters.  Suppose also that the Atkin--Lehner eigenspaces $S_2(\Gamma_0(N);\QQ)^\pm$ are simple $\TT$-modules.  Then
$$\mathrm{rank}\, J_0(N)(\QQ) = \dim S_2(\Gamma_0(N);\QQ)^+\pd$$ 
\end{prop}

\begin{rem}
The simplicity of $S_2(\Gamma_0(N);\QQ)^+$ and $S_2(\Gamma_0(N);\QQ)^-$ for large $N$ is only plausible if $N$ is prime.
\end{rem}

\begin{proof}
$J_0(N)$ is isogenous to $J_0(N)^+ \times J_0(N)^-$, where $J_0(N)_{\pm}$ is the abelian subvariety of $J_0(N)$ on which the Atkin-Lehner involution acts as $\pm 1$. By our assumption of simplicity, 
$$\TT(S_2(\Gamma_0(N);\QQ)) = K^+ \times K^-\cm$$ 
where $K_{\pm} = \End_{\QQ}(S_2(\Gamma_0(N);\QQ)_{\pm})$ are fields.  By BSD, $J_0(N)^+(\QQ)$ contains a non-torsion point.  The nonzero $\QQ$-vector space $J_0(N)^+(\QQ) \otimes \QQ$ is in fact a nonzero $K^+$-vector space.  Thus, 
$$\mathrm{rank}\, J_0(N)^+ = \text{nonzero multiple of } \left[ \dim_{\QQ} K^+  = \dim_{\QQ} S_2(\Gamma_0(N);\QQ)^+ \right]\pd$$  
But if we had $\mathrm{rank}\, J_0(N)^+ \geq 2 \dim_{\QQ} S_2(\Gamma_0(N);\QQ)^+$, then we would have
\begin{equation} \label{contradictsILS}
\liminf_{N \to \infty} \frac{\mathrm{rank}\, J_0(N)}{\dim_\QQ S_2(\Gamma_0(N);\QQ)} \geq  \liminf_{N \to \infty} \frac{\mathrm{rank}\, J_0(N)^+}{\dim_\QQ S_2(\Gamma_0(N);\QQ)} \geq 1.
\end{equation}
However, \eqref{contradictsILS} contradicts the result of Iwaniec--Luo--Sarnak \cite{ILS} that 
\begin{equation} \label{amazingILS}
\frac{\mathrm{rank}\, J_0(N)}{\dim_\QQ S_2(\Gamma_0(N);\QQ)} \leq c + o(1)
\end{equation}
for some explicit constant $c < 1$.  Therefore, we must have 
\begin{equation} \label{avoidcontradictionILS}
\mathrm{rank}\, J_0(N)^+ = \dim_\QQ S_2(\Gamma_0(N); \QQ)^+.
\end{equation}

Now, since $J_0(N)^-(\QQ) \otimes \QQ$ is a vector space over $K^-$, its rank must be a multiple of $\dim_\QQ S_2(\Gamma_0(N);\QQ)^-$.  If it is a nonzero multiple, then one would have
\begin{equation} \label{secondcontradictionILS}
 \mathrm{rank}\, J_0(N)^-(\QQ) \geq \dim_\QQ S_2(\Gamma_0(N);\QQ)^-.
\end{equation} 

Combining \eqref{secondcontradictionILS} with what we know from \eqref{avoidcontradictionILS} gives
\begin{align*}
\mathrm{rank}\, J_0(N) &= \mathrm{rank}\, J_0(N)^+ + \mathrm{rank}\, J_0(N)^- \\
&= \dim_{\QQ} S_2(\Gamma_0(N);\QQ)^+ + \mathrm{rank}\, J_0(N)^- \\
&\geq \dim_{\QQ} S_2(\Gamma_0(N);\QQ)^+ + \dim_\QQ S_2(\Gamma_0(N);\QQ)^- \\
&= \dim_{\QQ} S_2(\Gamma_0(N);\QQ),
\end{align*}
again contradicting the result of Iwaniec--Luo--Sarnak.  Therefore, \eqref{secondcontradictionILS} is false and 
\begin{equation} \label{avoidsecondcontradictionILS}
\mathrm{rank}\, J_0(N)^- = 0.
\end{equation}

Combining \eqref{avoidcontradictionILS} and \eqref{avoidsecondcontradictionILS} gives 
$$\mathrm{rank}\, J_0(N) = \mathrm{rank}\, J_0(N)^+ + \mathrm{rank}\, J_0(N)^- = \dim_\QQ S_2(\Gamma_0(N);\QQ)^+\pd$$\end{proof}

\begin{exam}
Referring back to Table 1 from the introduction, if the rate of convergence in the $o(1)$ term from \eqref{amazingILS} were sufficiently fast, Proposition \ref{heckeirreducibilityminimalistconjecture} would imply that $J_0(10007)$ has rank 379.
\end{exam}

\section{Detecting simple rational Hecke submodules via congruences with weight $1$ modular forms} \label{largegaloisorbits}This section is devoted to the proof of our main result.

\begin{proof}[Proof of Theorem \ref{largesimplesubmoduleviaweight1proof}]Let us recall the set-up:
\begin{itemize}\item We have a squarefree positive integer $N\equiv 3\mod 4$ and a prime $p\nmid N$;
\item We assume that the class number of $K=\QQ(\sqrt{-N})$ has a cyclic quotient $C$ of order $r$ with $p\nmid r$;
\item We denote by $m$ the order of $p$ in $(\ZZ/r\ZZ)^\times/\{\pm 1\}$, and we assume that $m>2$.\end{itemize}

To ease notation, set
\[V_p(N; R)=\left\{\begin{array}{ll}S_2(\Gamma_0(N);R)&\text{if $p=2$, and}\\
S_p(N,\varepsilon_N;R)&\text{if $p$ is odd.}\end{array}\right.\]
Our goal is to prove that $V_p(N;\QQ)$ contains a simple Hecke submodule of dimension at least $m$.

Since the characteristic polynomial of $T_n$ acting on $V_p(N;\QQ)$ and the characteristic polynomial of $T_n$ acting on $V_p(N;\ZZ_p)$ are identical and monic with coefficients in $\ZZ$, it is enough to show that there is a prime $\ell_0$ such that characteristic polynomial of $T_{\ell_0}$ acting on $V_p(N;\ZZ_p)$ has an irreducible factor of degree $\ge m$.


%
%
%
%
%
Let $\zeta_r$ be a fixed primitive $r$th root of unity in $\overline{\QQ}$, and let $\chi$ be the composition
$$\chi : \mathrm{Cl}\,K \twoheadrightarrow C \xrightarrow{g \mapsto \zeta_r} \overline{\QQ}^\times\cm$$
where $g$ is a generator of $C$. By class field theory, $\chi$ is associated with an everywhere-unramified character of $\mathrm{Gal}(\overline{\QQ}/K)$ that we also denote by $\chi$.  

Let $I_{\chi} := \mathrm{Ind}_K^\QQ \chi$ be the $2$-dimensional representation of $\mathrm{Gal}(\overline{\QQ}/\QQ)$ induced from $\chi$.  The Galois representation $I_{\chi}$ is associated with a weight 1 cusp form of level $N$ and nebentypus $\varepsilon_N$.  More precisely, there is a normalized cuspidal eigenform $f_{\chi} \in S_1(N,\varepsilon_N; \overline{\ZZ})$ with $q$-expansion $\sum_{n\ge 1}a_nq^n$ such that $T_nf_\chi=a_nf_\chi$ for all $n$ relatively prime to $N$ and such that the characteristic polynomial of $I_{\chi}(\frob_\ell)$ is $X^2 - a_\ell X + \varepsilon_N(\ell)$ for all primes $\ell\nmid N$. Calculating the trace of $I_\chi$ directly, we find that $a_\ell=\chi(\frob_\ell)+\chi^{-1}(\frob_\ell)$ whenever $\ell\nmid N$. By Chebotarev's density theorem, we may choose a prime $\ell_0\nmid Np$ such that $a_{\ell_0}=\zeta_r+\zeta_r^{-1}$.

Fix an embedding $\iota: \overline{\QQ} \rightarrow \overline{\QQ}_p$. The image of $f_\chi$ under the induced embedding $\iota:S_1(N,\varepsilon_N;\overline{\ZZ})\rar S_1(N,\varepsilon_N;\overline{\ZZ}_p)$, denoted $f_\chi^\iota$, is an eigenform in $S_1(N,\varepsilon_N;\overline{\ZZ}_p)$ satisfying $T_{n}(f_\chi^\iota)=\iota(a_n)f_\chi^\iota$ whenever $(n,Np)=1$.

Next we consider two maps that commute with the Hecke operators $T_n$ for all $(n,Np)=1$:
\begin{itemize}\item 
The reduction map
$$S_1(N,\varepsilon_N; \overline{\ZZ}_p) \otimes \overline{\FF}_p \xrightarrow{\mathrm{red}} S_1(N, \varepsilon_N; \overline{\FF}_p)$$
is injective, $\overline{\FF}_p$-linear, and Hecke-compatible in the sense that
\[\label{compatibilityreduction}
\mathrm{red}( T_n(h \otimes 1) ) = T_n(\mathrm{red}(h \otimes 1)) \text{ for all } h \in S_1(N,\varepsilon_N; \overline{\ZZ}_p)\]
\item Multiplication by the Hasse invariant $H_p\in M_{p-1}(\Gamma(1);\FF_p)$ yields an injective, Hecke-compatible map
\[[H_p]: S_1(N, \varepsilon_N; \overline{\mathbb{F}}_p) \rightarrow V_p(N; \overline{\FF}_p) :f \mapsto H_p f \]
\end{itemize}

%
%
%
%


It follows that $f = [H_p](\mathrm{red}(f_{\chi}^\iota\otimes 1)) \in V_p(N;\overline{\FF}_p)$ is nonzero and satisfies $T_n(f) = \overline{\iota (a_n)} f$ for all $n$ with $(n,Np) = 1$, the overline indicating reduction $\overline{\ZZ}_p\rar\overline{\FF}_p$. In particular,
\begin{equation}\label{lambda}T_{\ell_0}f=\lambda f\cm\quad\text{where $\lambda=\overline{\iota(\zeta_r)+\iota(\zeta_r^{-1})}$.}\end{equation}

Let $\mathfrak{m}$ denote the maximal ideal of the Hecke algebra $\mathbb{T}_{p,\overline{\FF}_p}$ associated with the eigenvalue system $\{ \overline{\iota(a_n)} \}_{n,(n,Np)=1}$.  The Galois representation $\overline{\rho}$ associated with $\mathfrak{m}$ equals $\overline{\iota I_{\chi}}$, i.e., the composition of $I_{\chi}$ with the ring homomorphism $\overline{\iota}:\overline{\ZZ}\rar\overline{\FF}_p$.  
\begin{itemize}
\item
Because $r > 2, \overline{\rho}$ is absolutely irreducible; indeed, its image is a nonabelian dihedral group acting in ``the usual way" on a 2-dimension $\overline{\FF}_p$-vector space.  
\item
$\overline{\rho}$ is not induced from $K_0 = \QQ(\sqrt{-1})$ or $\QQ(\sqrt{-3})$: 
%
To verify this, let $\overline{\tau} = \mathrm{Ind}_{K_0}^\QQ(\chi_0)$ for some $1$-dimensional representation $\chi_0: G_{K_0} \rightarrow \overline{\FF}_p^\times$.  The fields $K_0$ and $\QQ(\sqrt{-N})$ are linearly disjoint because $N \neq 1,3$.  By Chebotarev, we can find a rational prime $v$ which is unramified for $\overline{\rho}$, $K_0$, and $\QQ(\sqrt{-N})$, and for which (a.)\ $\frob_v$ is the nontrivial element of $\mathrm{Gal}(K_0 / \QQ)$ and (b.)\ $\frob_v$ is the trivial element of $\mathrm{Gal}(\QQ(\sqrt{-N}) / \QQ)$.  By (b.), $\overline{\rho}(\frob_v)$ is a scalar matrix.  On the other hand, $\overline{\tau}(\frob_v)$ is a non-scalar matrix by (a.).  Therefore, $\overline{\rho}$ cannot be isomorphic to $\overline{\tau}$.
\end{itemize}
By \cite[Proposition 4.2]{Khare}, $\mathfrak{m}$ is the image of a maximal ideal $\widetilde{\mathfrak{m}}$ for $\TT(V_p(N;\overline{\ZZ}_p))$ and the reduction map
$$V_p(N; \overline{\ZZ}_p)_{\widetilde{\mathfrak{m}}} \otimes \overline{\FF}_p \twoheadrightarrow V_p(N ; \overline{\FF}_p)_\mathfrak{m}$$ is surjective and Hecke-compatible. Therefore, for all $\ell\nmid Np$, the element $\overline{\iota(a_\ell)}\in\overline{\FF}_p$ is a root of the characteristic polynomial of $T_\ell$ acting on $V_p(N; \ZZ_p) \otimes \FF_p$, which is identical to the characteristic polynomial of $T_\ell$ acting on $V_p(N; \overline{\ZZ}_p) \otimes\overline{\FF}_p$.

Finally, the element $\lambda\in\overline{\FF}_p$ defined earlier in \eqref{lambda} generates an extension of $\FF_p$ having degree equal to the order of $\frob_p\in\mathrm{Gal}(\QQ(\zeta_r+\zeta_r^{-1})/\QQ)$ (note that $p$ is unramfied since $p\nmid r$). We have $\frob_p=p$ under the identification
\[\mathrm{Gal}(\QQ(\zeta_r+\zeta_r^{-1})/\QQ)=(\ZZ/r\ZZ)^\times/\{\pm1\}\cm\] so we may deduce that $[\FF_p(\lambda):\FF_p]=m$, where $m$ is the order of $p$ in $(\ZZ/r\ZZ)^\times/\{\pm1\}$.

It follows that the characteristic polynomial of $T_{\ell_0}$ acting on $V_p(N;\ZZ_p)\otimes\FF_p$ has an irreducible factor of degree $m$. Since the characteristic polynomial of $T_{\ell_0}$ acting on $V_p(N;\ZZ_p)\otimes\FF_p$ is just the mod-$p$ reduction of the characteristic polynomial of $T_{\ell_0}$ acting on $V_p(N;\ZZ_p)$, this latter polynomial must have an irreducible factor of degree $\ge m$, as desired.\end{proof}

\begin{exam}Let $N=719$ (this is prime). The field $\QQ(\sqrt{-719})$ has class number $31$, and of course, $2$ has order $5$ modulo $31$.

The space $S_1(719,\varepsilon;\QQ)$ is $15$-dimensional and Hecke-simple. The characteristic polynomial of $T_2$ acting on this space is irreducible over $\QQ$, and it factors as
\begin{equation}\label{example719eq1}(X^5 + X^2 + 1) (X^5 + X^3 + 1) (X^5 + X^4 + X^3 + X^2 + 1)\end{equation}
over $\FF_2$. Our method now yields the inequality
\[\dim_{\FF_2,\TT}^*S_2(\Gamma_0(719);\FF_2)\ge\dim_{\FF_2,\TT}^*S_1(N,\tilde{\varepsilon};\FF_2)\ge 5\cm\]
and, by lifting to characteristic zero, we obtain $d_2(719)=\dim^*S_2(\Gamma_0(719);\QQ)\ge 5$.

The characteristic polynomial of $T_2$ acting on $S_2(\Gamma_0(719);\QQ)$ factors as
\begin{multline}\label{example719eq2}(X^5 + 2X^4 - 4X^3 - 9X^2 - 2X + 1)\\(X^{10} - 11X^8 + 2X^7 + 39X^6 - 12X^5 - 52X^4 + 16X^3 + 24X^2 - 5X - 1)\cdot P\cm\end{multline}
where $\deg P=45$.

Note that the first factors in \eqref{example719eq1} and \eqref{example719eq2} are congruent modulo $2$. That is, in this particular case, the method of Theorem \ref{largesimplesubmoduleviaweight1proof} detects a simple rational Hecke submodule of dimension exactly equal to the lower bound $m$.

The other two factors in \eqref{example719eq1} divide the reduction of $P$ modulo $2$. A dimension argument shows that $\ker P(T_2)=S_2(\Gamma_0(719);\QQ)^+$. Thus, of the three factors in \eqref{example719eq1}, the first comes from $S_2(\Gamma_0(719);\QQ)^-$ while the second and third come from $S_2(\Gamma_0(719);\QQ)^+$.

\end{exam}

\begin{exam}One advantage of our method is that it allows us to derive concrete lower bounds on $d_2(N)$ without having to compute $S_2(\Gamma_0(N);\QQ)$ directly. Below, we give the bounds obtained by Theorem \ref{largesimplesubmoduleviaweight1proof} for some relatively high levels:

\begin{center}\begin{tabular}{|r|c|r|r|}\hline $N$ & $\mathrm{Cl}\,\QQ(\sqrt{-N})$ & $r$ & $m$\\\hline\hline
$81799$ & $C_{127}$ & $127$ & $7$\\\hline
$81839$ & $C_{377}$ & $377$ & $42$\\\hline
$81847$ & $C_{183}$ & $183$ & $60$\\\hline
$81883$ & $C_{35}$ & $35$ & $12$\\\hline
$81899$ & $C_{101}$ & $101$ & $50$\\\hline
\end{tabular}\smallskip

{\sc Table 2.} Bounds on $d_2(N)$ for large $N$.

\end{center}

Above, $C_j$ is a cyclic group of order $j$, $r$ is the order of the largest cyclic quotient (all class groups were cyclic in this range), and $m$ is the order of $2$ in $(\ZZ/r\ZZ)^\times/\{\pm1\}$. By Theorem \ref{largesimplesubmoduleviaweight1proof}, $d_2(N)\ge m$. This allows us to conclude, for example, that the $6820$-dimensional space $S_2(\Gamma_0(81839);\QQ)$ admits a simple Hecke submodule of dimension $\ge 42$.

\begin{rem} \label{Mersenneprimes}
Suppose $v = 2^p - 1$ is a Mersenne prime.  
Assuming Soundararajan's Conjecture \cite[(C1)]{Sound}, there will exist some $N$ for which $\QQ(\sqrt{-N})$ has class number $v$. This hypothetical integer $N$ would have size of magnitude roughly $2^{2p}$, up to factors that are polynomial in $p$.  On the other hand, the order of 2 in $(\ZZ/v\ZZ)^\times/ \{\pm 1\}$ equals $p$, which has size roughly $\frac{1}{2} \log N$.

In the range of the example above, the unconditional bound of Corollary \ref{logloglowerboundsmallsplitprime} is $d_2(N)\ge \lceil \log_2 \log_2 \left(\frac{N}{8} \right) \rceil=4$, and this bound applies when $N\equiv 7\mod 8$, as in the first three rows of Table 2. The $m$ of the first row is so small because the class number happens to be a Mersenne prime.
\end{rem}
\end{exam}

\section{Almost-simplicity of $S_2(\Gamma_0(N);\QQ)$}

\subsection{Small rational Hecke modules for $\Gamma_0(N)$} \label{almostsimplicity}

If $R$ is a ring and $f\in R[X]$, let $\deg^*_R(f)$ denote the largest degree of an irreducible factor of $f$ (in $R[X]$).

Let $\Phi_{N,\ell}^\pm(X)$ denote the characteristic polynomial of $T_\ell$ acting on $S_2(\Gamma_0(N);\QQ)^\pm$ and set
\begin{align*}\delta_{N,\ell}^\pm&=\deg_{\QQ}\Phi_{N,\ell}^\pm-\deg_{\QQ}^*\Phi_{N,\ell}^\pm\\&=\dim_{\QQ} S_2(\Gamma_0(N);\QQ)^\pm-\deg_{\QQ}^*\Phi_{N,\ell}^\pm\pd\end{align*}
$\delta_{N,\ell}^\pm$ captures the contribution to $\dim_{\QQ} S_2(\Gamma_0(N);\QQ)^\pm$ by {\it small} Hecke submodules. Conjecture \ref{fixedweighttrivialnebentypusincreasinglevel} implies that for any $\ell$ we have
\[\lim_{\substack{N\rar\infty\\\text{$N$ prime}}}\left(\frac{\delta_{N,\ell}^\pm}{\dim_{\QQ} S_2(\Gamma_0(N);\QQ)^\pm}\right)=0\pd\]
Empirically, much more appears true: $\delta_{N,\ell}^{\pm}$ seems to behave like a {\it bounded quantity}.  We computed $\delta_{N,2}^{\pm}$ for the $1092$ primes $N \in [100,9000]$:  

\begin{center}{\scriptsize\begin{tabular}{|r|r|r|r|r|r|r|r|r|r|r|r|r|r|r|r|r|r|r|r|}\hline $\delta$ & $0$ & $1$ & $2$ & $3$ & $4$ & $5$ & $6$ &$7$ & $8$ & $9$ & $10$ & $11$ & $12$ & $13$ & $14$ & $15$ &$16$ & $17$ & $18$\\\hline
freq.\ $\delta_{N,2}^+=\delta$ & 860 & 87 & 85 & 19 & 14 & 11 & 4 & 5 & 2& 3& 0 & 0 & 1 & 0 & 0 & 1 & 0 & 0 & 0 \\\hline
freq.\ $\delta_{N,2}^-=\delta$ & 865 & 103 & 68 & 29 & 15 & 3& 3& 1& 2& 1& 1& 0&0&0&0&0&0&0&1\\\hline

\end{tabular}}\smallskip

{\sc Table 3.} Histogram for values of $\delta_{N,2}^\pm$.\end{center}For example, $\delta_{N,2}^+=3$ for exactly $19$ primes $N\in[100,9000]$.

%
%
%
%
%
%
%
%
%

\subsection{Remarks on the size of $\delta_{N,2}^\pm$}
We will now try to give a plausible explanation for why $\deg \delta_{N,2}^\pm$ is usually so small.

For every Hecke eigenform $f \in S_2(N;\overline{\QQ})$, let $I_f$ denote the kernel of the homomorphism $\mathbb{T} \rightarrow \CC$ induced by the action of the prime to $N$ Hecke operators on $f$.  Let $\mathcal{O}$ denote the $\mathrm{Gal}(\overline{\QQ} / \QQ)$-orbit of $f$.  Note that $I_f$ depends only on $\mathcal{O}$.  The Hecke algebra $\mathbb{T}$ acts, by Picard functoriality, on the Jacobian $J_0(N)$ of the modular curve $X_0(N)$.  Let $A_{\mathcal{O}}:= J_0(N) / I_f$, the so-called \emph{optimal quotient} associated with $f$.  
   
The abelian variety $J_0(N)$ is isogenous to $\prod_{\mathcal{O}} A_{\mathcal{O}}$, where $\mathcal{O}$ ranges over all the $\mathrm{Gal}(\overline{\QQ} / \QQ)$-orbits of normalized Hecke eigenforms in $S_2(N;\overline{\QQ})$.  The dimension of $A_{\mathcal{O}}$ equals $g_{\mathcal{O}} = \# \mathcal{O}$. Furthermore, $A_{\mathcal{O}}$ has $\GL_2$-type for the field $K_{\mathcal{O}}$ generated by the Hecke eigenvalues of one orbit representative, i.e. $\deg \left( K_{\mathcal{O}} / \QQ \right) = g_{\mathcal{O}}$ and the endomorphism algebra of $A_{\mathcal{O}}$ equals $\mathrm{Mat}_2(K_{\mathcal{O}})$.  Because all Hecke eigenforms are new at prime level $N$, the abelian variety $A_{\mathcal{O}}$ has conductor $N^d$.  Furthermore, $A_{\mathcal{O}}$ admits a canonical polarization whose degree $d_{\mathcal{O}}$ is studied in \cite{Emerton}.   

Conversely, let $A / \QQ$ be an abelian variety of dimension $d$ which is $\GL_2$-type with respect to a totally real number field $F / \QQ$ of degree $d$.  Suppose that $A$ has conductor $N^d$.  Thanks to Khare--Winterberger's proof of Serre's Conjecture and the current state of knowledge for modularity lifting theorems, every such $A$ is isogenous to $A_{\mathcal{O}}$ for some $\mathrm{Gal}(\overline{\QQ} / \QQ)$-orbit of Hecke eigenforms in $S_2(N;\overline{\QQ})$.     

Every abelian variety $A_{\mathcal{O}}$ defines a rational point in $\mathcal{A}_{g_{\mathcal{O}}, d_{\mathcal{O}}}(\QQ)$, where $\mathcal{A}_{g,d}$ denotes the moduli space of $g$-dimensional abelian varieties equipped with a polarization of degree $d$.  These varieties tend to be general type for $g$ and $d$ large.  For example, it is known that $\mathcal{A}_{g,1}$ is general type for $g > 6$ and is not for $g < 6$.  

The Bombieri--Lang Conjectures suggest that $V(\QQ)$ is ``very thin" for $V / \QQ$ a variety of general type.  As explained in the previous paragraph, every (non-repeated) irreducible factor of $\delta_{N,2}^\pm$ of degree $g$ gives rise to an element of $\mathcal{A}_{g,d}(\QQ)$ for some (controlled) $d$.  This perhaps explains why the counts in Table 3 apparently ``thin out'' around $\delta_{N,2}^\pm = 6$.

\section{Detecting simple rational Hecke modules for $\Gamma(1)$ via congruences} 

Though we have so far only been able to derive explicit bounds in the cases covered by Theorem \ref{largesimplesubmoduleviaweight1proof}, congruences between spaces of modular forms can potentially be used to study other families of Hecke modules. Interestingly, it may be possible to apply congruence methods toward the conjectural simplicity of spaces of level $1$ cusp forms.

There is a well-known Hecke-compatible congruence between $S_{p+1}(\Gamma(1), \FF_p)=S_{p+1}(\SL_2(\ZZ);\FF_p)$ and $S_2(\Gamma_0(p);\FF_p)$ \cite[Th\'{e}or\`{e}me 11.b]{Serrecong}. By an argument similar to the one made in the proof of our main theorem, this congruence implies
\begin{equation}\label{alpha}\dim_{\QQ,\TT}^* S_{p+1}(\Gamma(1);\QQ)\ge\dim_{\FF_p,\TT}^* S_{2}(\Gamma_0(p);\FF_p)\pd\end{equation}
Let $D_p$ be the common dimension of $S_{p+1}(\Gamma(1))$ and $S_2(\Gamma_0(p))$ over both $\QQ$ and $\FF_p$, and let
\[\alpha_p=\frac{1}{D_p}\dim_{\FF_p,\TT}^*S_2(\Gamma_0(p);\FF_p)\pd\]
The number $\alpha_p\in[0,1]$ measures ``how close'' the inequality \eqref{alpha} comes to proving the Hecke-simplicity of $S_{p+1}(\Gamma(1);\QQ)$ implied by Maeda's conjecture.

Note that since $\dim_{\FF_p,\TT}^*S_2(\Gamma_0(p);\FF_p)\le\dim_{\QQ,\TT}^*S_2(\Gamma_0(p);\QQ)$, the existence of the Atkin--Lehner involution on $S_2(\Gamma_0(p);\QQ)$ implies that $\limsup_{p\rar\infty}(\alpha_p)\le\tfrac{1}{2}$. Nonetheless, if one could show that $\{\alpha_p\}_p$ is bounded below by some threshold with high probability, this would constitute considerable progress towards Maeda's conjecture. If one is willing to speculate that the decompositions of $S_2(\Gamma_0(p);\FF_p)^\pm$ into simple $T_\ell$-modules resemble the cycle decompositions of random permutations, this approach does not seem totally unrealistic.

\begin{conj}[Imprecise form]\label{randompermutationmodelimprecise}Let $D_p^\pm=\dim_{\QQ}S_2(\Gamma_0(p);\QQ)^\pm$ and fix a prime $\ell$.

The factorization type of the characteristic polynomial of $T_\ell$ acting on $S_2(\Gamma_0(p);\FF_p)$ is distributed like the cycle type of $\sigma_+\sqcup\sigma_-$ where $\sigma_+$ and $\sigma_-$ are independently chosen random permutations on $D_p^+$ and $D_p^-$ elements, respectively.\footnote{More precisely, we believe that $\mathbb{T}_{\pm, \mathrm{red}}$, the reduced quotient of the image of $\mathbb{T}$ in $\End (S_2(\Gamma_0(p);\FF_p)_{\pm})$, behave like two independent \'{e}tale $\FF_p$-algebras distributed according to the random permutation model from \cite{Lip}.  The ``longest cycle of $T_2$" statistic studied in this subsection is very unlikely to differ substantially from the ``longest cycle of $\mathbb{T}_{\pm,\mathrm{red}}$.'' Any discrepancy between these two statistics would correspond to the projection of $T_2$ onto the largest simple factor of $\mathbb{T}_{\pm,\mathrm{red}}$ (an extension of $\FF_p$ of large degree) landing in a proper subfield.

}\end{conj}

%
%
%
%

%
%
%




We have tested Conjecture \ref{randompermutationmodelimprecise} numerically for $\ell=2$.  To put the results in context, we first recall a special function which describes the distribution of the longest cycle length in large random permutations.   

\begin{defn}
The Dickman function $\rho$, defined by $\rho(u)=1$ on $u\in(0,1]$ and by
\[\rho(u)=\frac{1}{u}\int_{u-1}^u\rho(t)\,dt\]
for $u>1$, captures the statistics of long cycles in random large permutations \cite{Granville}: the probability that a randomly chosen permutation on $N$ elements has longest cycle of length $\leq \frac{N}{u}$ approaches $\rho(u)$ as $N \to \infty$.
\end{defn}

We define random variables $Y_\ell,Y_\ell^+,Y_\ell^-$ on the set of primes $\ne\ell$ as follows:
\[Y_{\ell}(p)=\frac{1}{D_p}\deg_{\FF_p}^*\widetilde{\Phi}_{p,\ell}(X)\quad\text{and}\quad Y_{\ell}^\pm(p)=\frac{1}{D_p^\pm}\deg_{\FF_p}^*\widetilde{\Phi}_{p,\ell}^\pm(X)\cm\]
where $\widetilde{\Phi}_{p,\ell}$ is the characteristic polynomial of $T_\ell$ on $S_2(\Gamma_0(p);\FF_p)$ and $\widetilde{\Phi}_{p,\ell}^\pm$ is the characteristic polynomial of $T_\ell$ on $S_2(\Gamma_0(p);\FF_p)^\pm$. Note that for any $p$
\[Y_{\ell}(p)=\max\{Y_{\ell}^+(p),Y_{\ell}^-(p)\}\pd\]

Assuming Conjecture \ref{randompermutationmodelimprecise}, we are led to guess that $Y_{\ell}^\pm\le y$ with probability approaching $\rho\bigl(\frac{1}{2y}\bigr)$.   If $Y_{\ell}^+$ and $Y_{\ell}^-$ are independent, we would have $Y_\ell\le y$ with probability approaching $\rho \bigl(\frac{1}{2y} \bigr)^2$.

Using SAGE, we computed $Y_2(p)$ for all primes $p \in [100,10200]$. The results of this computation are displayed in Figure \ref{dickmanrho}: We have overlaid the CDF of the random variable $p\mapsto Y_2(p)$ (in light red dots), with the function $y\mapsto\rho\bigl(\frac{1}{2y}\bigr)^2$ (in blue). The agreement is striking. We conclude with the following conjecture:

\setcounter{figure}{3}

\begin{figure}
  \includegraphics[width=\linewidth]{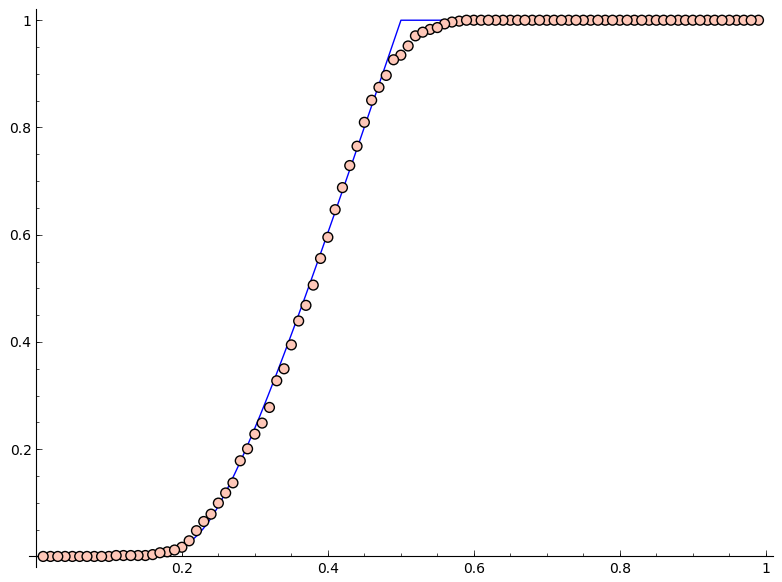}
  \caption{The CDF of $Y_2(p)$ for all primes $p \in [100,10200]$ (light red), versus $u \mapsto \rho \left(\frac{1}{2u} \right)^2$ (blue).}
  \label{dickmanrho}
\end{figure}


\begin{conj}[Precise form]\label{preciseform}
For any fixed prime $\ell$, 
\[\lim_{X \to \infty}\frac{\# \{ \text{primes } p \in [X,2X]: Y_\ell(p)  \leq y \}  }{\#  \{\text{primes } p \in [X,2X]  \}} = \rho \left( \frac{1}{2y} \right)^2.\]
where $\rho$ is the Dickman function.
\end{conj}

For example, Conjecture \ref{preciseform} would predict that for a randomly chosen large prime $p$, $S_{p+1}(\Gamma(1);\QQ)$ admits a simple Hecke submodule of dimension at least $\ge\frac{1}{8}\dim_\QQ S_{p+1}(\Gamma(1);\QQ)$ with probability $\ge 0.99997$.

\section*\thanks{{\bf Acknowledgements.} We would like to thank Peter Sarnak and Akshay Venkatesh for their encouragement. We are specifically grateful to Akshay Venkatesh for pointing out how to obtain the class number bound in the proof of Corollary \ref{logloglowerboundsmallsplitprime}.}

\end{document}